\documentclass[11pt,reqno]{amsart}
\usepackage[all]{xy}
\usepackage{amssymb}
\usepackage{amsthm}
\usepackage{cite}
\usepackage[normalem]{ulem}
\usepackage{amsmath,mathtools}
\usepackage{amscd,enumitem}
\usepackage{caption}
\usepackage[justification=centering]{caption}
\usepackage{verbatim}
\usepackage{eurosym}
\usepackage{float}
\usepackage{color}
\usepackage{dcolumn}
\usepackage[mathscr]{eucal}
\usepackage[all]{xy}
\usepackage{bbm}
\usepackage{ytableau}
\usepackage{soul}
\usepackage[textheight=8.7in, textwidth=6.6in]{geometry}
\newtheorem*{conj*}{Conjecture}
\newtheorem{theorem}{Theorem}[section]

\theoremstyle{definition}
\newtheorem*{remark}{Remark}

\theoremstyle{plain}

\theoremstyle{definition}
\newtheorem*{question}{Question}

\theoremstyle{plain}

\newtheorem{corollary}[theorem]{Corollary}

\newcommand{\ord}{\mathrm{ord}}

\newcommand{\SL}{\operatorname{SL}}

\renewcommand{\pmod}[1]{\,\,({\rm mod}\,\,{#1})}

\numberwithin{equation}{section}

\newtheoremstyle{example}
  {5pt}   % ABOVESPACE
  {5pt}   % BELOWSPACE
  {\normalfont}  % BODYFONT
  {0pt}       % INDENT (empty value is the same as 0pt)
  {\bfseries} % HEADFONT
  {.}         % HEADPUNCT
  {5pt plus 1pt minus 1pt} % HEADSPACE
  {}          % CUSTOM-HEAD-SPEC
\theoremstyle{example}
\newtheorem*{example}{Example}

\newtheoremstyle{definition}
  {3pt}   % ABOVESPACE
  {3pt}   % BELOWSPACE
  {\normalfont}  % BODYFONT
  {0pt}       % INDENT (empty value is the same as 0pt)
  {\bfseries} % HEADFONT
  {.}         % HEADPUNCT
  {5pt plus 1pt minus 1pt} % HEADSPACE
  {}          % CUSTOM-HEAD-SPEC
\theoremstyle{definition}

\def\({\left(}
\def\){\right)}

\usepackage{centernot}

\usepackage{comment}
\usepackage{xcolor}
\usepackage{ytableau}

\begin{document}
\title{On Certain McKay Numbers of Symmetric Groups}


\author{Annemily G. Hoganson}
\address{Department of Mathematics and Statistics, Carleton College,
Northfield, MN 55057}
\email{hogansona@carleton.edu}
% https://orcid.org/0000-0002-0288-4095

\author{Thomas Jaklitsch}
\address{Department of Mathematics, University of Virginia, Charlottesville, VA 22904}
\email{gvs3ka@virginia.edu}

\thanks{}

\begin{abstract} 
For primes $\ell$ and nonnegative integers $a$, we study the partition functions $$p_\ell(a;n):= \#\{\lambda \vdash n : \ord_\ell(H(\lambda))=a\},$$ where $H(\lambda)$ denotes the product of hook lengths of a partition $\lambda$. These partition values arise as the McKay numbers $m_\ell(\ord_\ell(n!) - a; S_n)$ in the representation theory of the symmetric group. We determine the generating functions for $p_\ell(a;n)$ in terms of $p_\ell(0;n)$ and specializations of specific D'Arcais polynomials. For $\ell = 2$ and $3$, we give an exact formula for the $p_\ell(a;n)$ and prove that these values are zero for almost all $n$. For larger primes $\ell$, the $p_\ell(a;n)$ are positive for sufficiently large $n$. Despite this positivity, we prove that $p_\ell(a;n)$ is almost always divisible by $m$ for any integer $m$. Furthermore, with these results we prove several Ramanujan-type congruences. These include the congruences $$p_\ell(a;\ell^k n - \delta(a,\ell)) \equiv 0 \pmod{\ell^{k+1}},$$ for $0<a< \ell$, where $\ell = 5, 7, 11$ and $\delta(a,\ell) := (\ell^2 - 1)/24 + a\ell$, which answer a question of Ono.
\end{abstract}

\maketitle

\section{Introduction and Statement of Results}

A \emph{partition} of an integer $n$ is a non-increasing sequence of positive integers that sum to $n$. The partition function $p(n)$, which counts the number of partitions of $n$, has been widely studied. Some of the earliest significant results were the work of Ramanujan. In particular, Hardy and Ramanujan illustrated the function's rapid growth  \cite{Andrews}, proving that as $n\to \infty$, $$ p(n) \sim \frac{1}{4n\sqrt{3}}e^{\pi\sqrt{2n/3}}.$$

In addition to describing the growth of $p(n)$, Ramanujan investigated arithmetic properties of the partition function and proved the famous congruences %\cite{BerndtRam}:

\begin{equation}\label{eq: Ram congruences}
    \begin{split}
        p(5n+4) &\equiv 0 \pmod 5,\\
        p(7n+5) &\equiv 0 \pmod 7,\\
        p(11n+6) &\equiv 0 \pmod{11}. %\\
    %p(25n +24) &\equiv 0 \pmod{25}\\
    %p(49n+47) &\equiv 0 \pmod{49}\\
    %p(121n+116) &\equiv 0 \pmod{121}
    \end{split}
\end{equation}

\noindent He further conjectured that these congruences could be extended to arbitrary powers of $5,7$ and $11$. This was subsequently proven by Atkin \cite{Atkin} and Watson \cite{Watson} (with a slight modification for $7$). Following these elegant results, little progress was made on describing the congruence properties of $p(n)$ until Ahlgren and Ono proved congruences for arbitrary moduli $m$ coprime to 6 \cite{OnoDis, AhlgOno}. However, these congruences are not nearly as systematic as those proven for powers of $5$, $7$, and $11$. For example, the simplest arithmetic progression which gives a congruence mod $17$ is $$p(48037937\cdot n + 1122838) \equiv 0 \pmod{17}.$$

Given the rarity of these congruences, it is natural to study the distribution of the partition function over residue classes modulo a prime $\ell$. In the case where $\ell = 2$, Parkin and Shanks conjectured that the proportion of $p(n)$ which are even (resp. odd) is $1/2$ \cite{ParkinSh}. While their conjecture is easily stated, even the simpler problem of whether either residue class modulo $2$ contains a positive density of $p(n)$ remains open. Moreover, while it has been shown that there are infinitely many $p(n)$ in each residue class modulo 2 and all primes $\ell \ge 5$ \cite{atkinNewsC, ahlgrenNewsC, kloveNewsC, kloveNewsC2, kolbergNewsC, newmanNewsC, OnoDis}, it is still unknown whether there are infinitely many $p(n)$ that are multiples of 3. As these unresolved questions show, it is difficult to understand the congruence properties of $p(n)$ directly. An alternative approach would be to divide the partitions of a given $n$ into subcollections whose congruence properties are easier to understand. 

One place to look for such a decomposition is in the representation theory of the symmetric group. Although partitions may seem purely combinatorial at first glance, $p(n)$ arises in this context as the number of irreducible representations of $S_n$. In addition, the structure of the partitions of $n$ can be used to compute these representations. For a partition $\lambda \vdash n$, the action of $S_n$ on the Young tableaux of $\lambda$ produces an irreducible representation $\rho_\lambda: S_n \to \text{GL}(V_\lambda)$ \cite{JamesKerber}. Here $V_\lambda$ is a vector space with dimension given by the famous Frame-Thrall-Robinson hook length formula. 

To make this precise, we recall that the Ferrers-Young diagram of a partition ${\lambda=\{\lambda_1 \geq \lambda_2 \dots \geq \lambda_t\}}$ depicts $\lambda$ as a left-justified two-dimensional grid with $\lambda_k$ boxes in the $k$th row. We denote the box in row $i$ and column $j$ by $(i,j)$. The \emph{hook} of $(i,j)$ consists of the box $(i,j)$ as well as the boxes directly below and directly to the right of $(i,j)$. The \emph{hook length}, $H(i,j)$, is the number of boxes in the hook of $(i,j)$. The \emph{hook product} of $\lambda$ is then 
\begin{equation}
    H(\lambda) := \prod_{(i,j)}H(i,j).
\end{equation}
The Frame-Thrall-Robinson hook length formula then connects the hook product of $\lambda$ to the representation $\rho_\lambda$ by the equation $\dim(V_\lambda) = \frac{n!}{H(\lambda)}$ \cite{FrameThR}.

\begin{example}
    The partition $\lambda=\{4,2,1\}$ has the following Ferrers-Young diagram with the hook lengths shown in each box:
    $$\begin{ytableau}
    6 & 4 & 2 & 1\cr
    3 & 1\cr
    1 \cr
    \end{ytableau}$$
In this case, we have $H(\lambda) = 144$.
\end{example}

Considering the centrality of the hook product in representation theory, one might naturally try to divide the partitions of $n$ according to the properties of their corresponding hook products. One approach is to fix a prime $\ell$ and divide the partitions of $n$ by the $\ell$-adic valuations of their hook products. To this end, we define the family of functions $\{p_\ell(a;n)\}_{a \geq 0}$, where 
\begin{equation}p_\ell(a;n):= \#\{\lambda \vdash n : \ord_\ell(H(\lambda))=a\}.\end{equation}
Since each partition has a hook product and each hook product has an $\ell$-adic valuation, we have $$p(n) = p_\ell(0;n) + p_\ell(1; n) + p_\ell(2;n) + \ldots $$
It turns out that $p_\ell(a;n)$ counts the number of irreducible representations of $S_n$ such that 

\noindent${\ell^{\ord_\ell(n!)-a} || \dim(V_\lambda)}$ \cite{JamesKerber}. Note that this is the classical McKay number $m_\ell(\ord_\ell(n!) - a; S_n)$ \cite{Nakamura}.

In this paper, for each prime $\ell$, we investigate the divisibility properties of $p_\ell(a;n)$ modulo any integer $m$. To do this, we work with the generating function of $p_\ell(a;n)$, which we denote by \begin{equation}P_\ell(a;q) := \sum_{n = 0}^\infty p_\ell(a;n) q^n.\end{equation}
As a power series in $x$, Nakamura assembles all of these generating functions as an infinite product expansion for $\sum_{a = 0}^\infty P_\ell(a;q)x^a$ (see Section \ref{sec: Nakamura} and \cite{Nakamura}). 

In order to use Nakamura's formula to study the $P_\ell(a;q)$ individually for each $a$, we consider the $q$-expansion
\begin{equation} \prod_{m=1}^\infty (1-q^m)^r  = \sum_{n=0}^\infty f(r; n)q^n = 1 - r q + \frac{r(r-3)}{2}q^2 - \frac{r(r-1)(r-8)}{6}q^3+\ldots\end{equation}
The coefficients $f(r;n)$ are known as the D'Arcais polynomials and have been studied by Newman and Serre \cite{DArcais, Newman, SerreNew}. The Nekrasov-Okounkov Hook Length Formula provides closed formulas for the D'Arcais polynomials \cite{NO} and in order to illuminate certain divisibility properties of these polynomials, we provide alternative closed formulas in Section \ref{sec: D'Arcais}. We utilize the D'Arcais polynomials to reduce the description of $P_\ell(a;q)$ to that of the single generating function $P_\ell(0;q)$. Moreover, we use these alternative formulas to study the congruence properties of the $p_\ell(a;n)$.

To make these precise, we first fix some notation. If $k$ and $a$ are nonnegative integers and $\ell$ is prime, let $[k]_\ell := \frac{\ell^k-1}{\ell-1}$ and define
\begin{equation}\label{Zdef}\mathcal{Z}_\ell(a) := \bigg\{(z_1, z_2,...) \in \bigoplus_{i\in \mathbb{N}}\mathbb{Z}_{\ge 0} \; :\;\; \sum_{k\ge 1}z_k [k]_\ell = a \bigg\}.\end{equation} Note that $\mathcal{Z}_\ell(a)$ is a finite set. We further define
\begin{equation}\Omega_\ell(a;q) := \sum_{\vec{z} \in \mathcal{Z}_\ell(a)} \prod_{k=1}^\infty q^{z_k \ell^k} \cdot \sum_{m=0}^{\lfloor z_k/\ell \rfloor}f\left(\ell^{k+1};m\right)\cdot f\left(-\ell^k; z_k - m\ell\right). \end{equation}

\noindent We then have the following explicit description of $P_\ell(a;q)$ in terms of $P_\ell(0;q)$ and this algebraic expression involving D'Arcais polynomials.

\begin{theorem}\label{thm:general pl}
If $\ell$ is prime and $a$ is a nonnegative integer, then we have
$$P_\ell(a;q) = \Omega_\ell(a;q) \cdot P_\ell(0;q).$$

\end{theorem}

Theorem \ref{thm:general pl} shows that $p_\ell(a;n)$ is a simple algebraic combination of specializations of D'Arcais polynomials and values $p_\ell(0;m)$ for a finite set of $m$ depending on $a$ and $n$. If $a < \ell$, Theorem \ref{thm:general pl} reduces to the following expression.

\begin{corollary}\label{thm:pl polys}
    If $\ell$ is prime and $a$ is an integer such that $0 \le a < \ell$, then $$ P_\ell(a;q) = f(-\ell;a)q^{a\ell} \cdot P_\ell(0;q).$$
\end{corollary}

\begin{example}
    If $0\le a< \ell$, then $p_\ell(a;n)$ is the product of $p_\ell(0;n-a\ell)$ and a polynomial in $\ell$.
    For $a = 1, 2,3$ and $\ell \ge 5$, Corollary \ref{thm:pl polys} gives
    \begin{align*}
        p_\ell(1;n) &= \ell \cdot p_\ell(0; n-\ell), \\
        p_\ell(2;n) &= \frac{\ell(\ell+3)}{2} \cdot p_\ell(0; n - 2\ell),\\
        p_\ell(3; n) &= \frac{\ell(\ell+1)(\ell+8)}{6} \cdot p_\ell(0; n - 3\ell).
    \end{align*}
    These formulas illustrate the simplicity of this relationship for $0 \le a<\ell$.
\end{example}

Reducing the study of $P_\ell(a;q)$ to that of $P_\ell(0;q)$ allows us to use properties of $p_\ell(0;n)$, which is a well-studied partition function. It counts the number of $\ell$-core partitions of an integer $n$ (partitions where $\ell$ does not divide any hook length). In the cases when $\ell = 2$ or $3$, the generating functions for $p_\ell(0;n)$ are given by (see p. 334 of \cite{GranOno}) \begin{equation} \label{2core}
    P_2(0;q) = \sum_{m=0}^\infty q^{\frac{m^2+m}{2}},
\end{equation}  \begin{equation} \label{3core} P_3(0;q) = \sum_{m=0}^\infty  \left(\sum_{d | 3m+1}\left(\frac{d}{3}\right) \right) q^m. \end{equation} These expressions yield the following formulas.

\begin{theorem}\label{thm: 2-3 formulas}
    The following are true:
    
    \begin{enumerate}
        \item For all $n$, we have $$p_2(1; n) =  \begin{cases}2 &\text{if } n-2 = \frac{m(m+1)}{2} \;\text{for some } m\in \mathbb{N},\\ 0 &\text{otherwise.}\end{cases}$$
        \item For $n>2$, we have $$p_3(1;n) = 3\sum_{d | 3n - 8}\left(\frac{d}{3}\right).$$
        \item For $n>5,$ we have $$ p_3(2;n) = 9\sum_{d | 3n-17}\left(\frac{d}{3}\right).$$
    \end{enumerate}
\end{theorem}

\begin{remark}
Note that it is possible to obtain formulas for $p_2(a;n)$ and $p_3(a;n)$ for any nonnegative integer $a$ using Theorem \ref{thm:general pl}, (\ref{2core}), and (\ref{3core}). However, these formulas become significantly complex as $a$ becomes larger. 
\end{remark}

\begin{question}
Are there combinatorial proofs of Theorem \ref{thm:general pl} and the preceding theorem?
\end{question}

\bigskip
\noindent We use these formulas to prove the following results on the paucity of nonzero values of these $p_\ell(a;n)$.

\begin{corollary}\label{thm:density23}
    If $a$ is a nonnegative integer, then we have
    $$\#\{n \le X : p_2(a;n) \not= 0 \} = O_a\left(\sqrt{X}\right)
    $$
    and 
    $$\#\{n \le X : p_3(a;n) \not= 0 \} = O_a\left(\frac{X}{\sqrt{\log X}}\right).
    $$
\end{corollary}

\begin{example}
To illustrate the results of Corollary \ref{thm:density23}, we provide numerical data for the cases ${a=0}$ and $X = 10^k$, where $1 \le k \le 8$. 
To this end, define $$\gamma_\ell(a; X) := \frac{\#\{1\le n \le X : p_\ell(a;n) \ne 0\}}{X}. $$
The table below gives the proportion of $p_\ell(0;n) \ne 0$ for $\ell = 2,3$. 
\begin{center}
\begin{table}[H]
\begin{tabular}{ |c|c|c|c|c| } 
 \hline
 X & $\gamma_2(0;X)$ & $\gamma_3(0; X)$\\ \hline
 $10^1$ & $.40000$ & $.80000$ \\
 $10^2$ & $.13000$ & $.57000$ \\
 $10^3$ & $.04400$ & $.47400$ \\ 
 $10^4$ & $.01400$ & $.41340$ \\
 $10^5$ & $.00446$ & $.37064$ \\
 $10^6$ & $.00141$ & $.33893$ \\
 $10^7$ & $.00044$ & $.31424$ \\
 $10^8$ & $.00014$ & $.29431$ \\
\hline
\end{tabular}
\caption{$\gamma_\ell(0;X)$ for $\ell = 2,3$ and $X \le 10^8$}
\vspace{-6mm}
\end{table}
\end{center}
\end{example}

A natural next question is whether the vanishing of $p_\ell(a;n)$ generalizes to arbitrary primes. The work of Granville and Ono shows that this is not the case \cite{GranOno}. In particular, they proved that the number of $k$-core partitions of any integer $n$ is positive for all integers $k\ge 4$. Although the $p_\ell(a;n)$ do not vanish for primes $\ell \ge 5$, one might hope for vanishing mod $m$. In this vein, we prove that $p_\ell(a;n)$ is divisible by any integer $m$ asymptotically $100\%$ of the time for any prime $\ell$.

%16:34

\begin{theorem}\label{thm:density}
    If $\ell$ is an odd prime, $a$ is a nonnegative integer, and $m$ is a positive integer, then there exists a real number $\alpha_{\ell, m} > 0$ such that 
    $$\#\{n \le X : p_\ell(a;n) \not\equiv 0 \pmod{m} \} = \text{O}_{\ell,m, a}\left(\frac{X}{(\log X)^{\alpha_{\ell, m}}}\right).
    $$
    In particular $p_\ell(a;n) \equiv 0 \pmod{m}$ for almost all $n$.
\end{theorem}

We achieve this result by studying the modularity of $P_\ell(0;q)$ and combining a result of Serre \cite{Serre} with Theorem \ref{thm:general pl}.

\begin{example}
To illustrate the results of Theorem \ref{thm:density}, we provide numerical data for the case of $p_5(0;n)$. First, define $$\delta_\ell(a; m ;X) := \frac{\#\{1\le n \le X : p_\ell(a; n) \not\equiv 0 \pmod{m}\}}{X}.$$
We compute examples of $\delta_5(0;m;X)$ for $m = 2,3,4,5$ and $X = 10^k$ where $1 \le k\le 6$.
\begin{center}
\begin{table}[H]
\begin{tabular}{ |c|c|c|c|c| } 
 \hline
 X & $\delta_5(0; 2; X)$ & $\delta_5(0; 3; X)$ & $\delta_5(0; 4; X)$ & $\delta_5(0; 5; X)$ \\ \hline
 $10^1$ & $.6000$ & $.7000$ & $.9000$ & $.7000$ \\
 $10^2$ & $.2300$ & $.4700$ & $.5600$ & $.6300$ \\ 
 $10^3$ & $.0760$ & $.3640$ & $.3510$ & $.5650$ \\
 $10^4$ & $.0244$ & $.3040$ & $.2424$ & $.5224$ \\
 $10^5$ & $.0077$ & $.2667$ & $.1790$ & $.4923$ \\
 $10^6$ & $.0024$ & $.2404$ & $.1399$ & $.4690$ \\
 \hline
\end{tabular}
\caption{$\delta_5(0; m; X)$ for $m = 2, 3, 4, 5$ and $X\le 10^6$.}
\vspace{-8mm}
\end{table}
\end{center}
Note that the rate of convergence depends on $m$, and is in general quite slow.
\end{example}

Given that $p_\ell(a;n)$ is almost always divisible by any integer $m$, a natural following question is to ask whether there are explicit congruences of the form $p_\ell(a;An+B) \equiv 0 \pmod{m}$. For the usual partition function, the most famous such congruences are the Ramanujan congruences. We recall in particular the Ramanujan congruences modulo arbitrary powers of $5$, $7$, and $11$ \cite{Atkin, Watson}:
\begin{equation}
    \begin{split}
        p(5^m n +\delta_{5,k}) &\equiv 0 \pmod{5^m},\\
        p(7^m n +\delta_{7,k}) &\equiv 0 \pmod{7^{[m/2]+1}},\\
        p(11^m n +\delta_{11,k}) &\equiv 0 \pmod{11^m},
    \end{split}
\end{equation}
where $\delta_{p,k}:=24^{-1}\pmod{p^k}$.

In \cite{OnoMcK}, Ono proved that if $\ell = 5,7,$ or $11$, and $m=1$ or $2$, then for all $a \geq 0$, $$p_\ell(a; \ell^m n - \delta(\ell)) \equiv 0 \pmod{\ell^m},$$ where $\delta(\ell) := (\ell^2 - 1)/24$. While these congruences do align with the congruences of $p(n)$ modulo $5^2, 7^2,$ and $11^2$, they cannot be extended to align with the congruences of $p(n)$ modulo $5^3, 7^3,$ and $11^3$. In particular, we have that
$$p_5(14;99) = 5594200 \equiv 75 \not\equiv 0 \pmod{5^3}.$$

In view of this example, Ono speculated that Ramanujan's congruences modulo powers of $5$, $7$, and $11$ have modified generalizations to the $p_\ell(a;n)$ functions. We show that this is indeed the case for $0<a\le\ell$. Moreover, we find similar congruences modulo $\ell^3$ and $\ell^4$ for $\ell < a < \ell(\ell+1)$. To this end, let ${\delta(a;\ell) := \frac{\ell^2-1}{24}-a \ell}$. Then we have the following.

\begin{theorem}\label{thm:congruences}
    If $\ell = 5,7$, or $11$, then for every nonnegative integer $n$, the following are true:
    \begin{enumerate}
        \item For $0 < a < \ell$ and $m \ge 1$, we have $$p_\ell(a; \ell^m n -\delta(a;\ell))\equiv 0 \pmod{\ell^{m+1}}.$$
        \item For $a =\ell$ and $m \ge 1$, we have $$p_\ell(\ell; \ell^m n -\delta(\ell;\ell))\equiv 0 \pmod{\ell^m}.$$
        \item For $\ell + 1 < a < 2\ell$, we have
        $$p_\ell(a; \ell^3 n -\delta(a;\ell))\equiv 0 \pmod{\ell^4}.$$
        \item For $a = 2\ell$ We have $$p_\ell(2\ell; \ell^4 n -\delta(2\ell;\ell))\equiv 0 \pmod{\ell^4}.$$
        \item For $0 < a < \ell(\ell+1)$, we have 
        $$p_\ell(a; \ell^3 n -\delta(a;\ell))\equiv 0 \pmod{\ell^3}.$$
    \end{enumerate}
\end{theorem}

\begin{example}
    We provide examples of each congruence given in Theorem \ref{thm:congruences}.
    \begin{enumerate}
        \item Let $\ell = 11$, $a=3$, and $m=2$. We have that $$p_{11}(3;121n + 28) \equiv 0 \pmod{11^3}.$$
        \item Let $\ell = a = 7$, and $m = 3$. We have that $$p_7(7; 343n + 47) \equiv 0 \pmod{7^3}.$$
        \item Let $\ell = 5$ and $a = 9$. We have that $$p_5(9;125n + 44) \equiv 0 \pmod{5^4}.$$
        \item Let $\ell = 5$, $a = 10$, and $m=3$. We have that $$p_5(10; 125n + 49) \equiv 0 \pmod{5^3}. $$
        \item Let $\ell = 7$ and $a = 45$. We have that $$p_7(45; 343n + 313) \equiv 0 \pmod{7^3}$$
    \end{enumerate}
\end{example}

The key difference in our result is that there is a shift in the arithmetic progression which depends on $a$. Note also that $\delta(\ell) = \delta(0;\ell)$ and that, for $0 < a < \ell$, Ono's congruences modulo $\ell^2$ occur in the first set of congruences in the $m=1$ progression where $n = k \ell -a$ for positive integers $k$. To obtain this result, we utilize Theorem \ref{thm:general pl} and the Ramanujan-type congruences for $p_\ell(0;n)$ when $\ell = 5,7,$ or $11$ modulo arbitrary powers of $\ell$ which are described in \cite{GarvanKimStanton}.

This paper is organized as follows. In Section 2, we state a theorem of Nakamura, which gives an infinite product expansion for the generating function of $p_\ell(a;n)$. To help us make use of Nakamura's product, we recall the formula for the D'Arcais polynomials from \cite{heimNeu}. In order to illuminate certain divisibility properties of special values of these polynomials, we require an alternative description that is also given in this section. We then review necessary background on the theory of modular forms in order to apply Serre's results in \cite{Serre} to the generating function of $p_\ell(0;n)$. In Section 3, we use Nakamura's product from Section \ref{sec: Nakamura} to prove Theorem \ref{thm:general pl} and Corollary \ref{thm:pl polys}. We combine these results with formulas for the $2$- and $3$-core partition functions (i.e. (\ref{2core}) and (\ref{3core})) to prove Theorem \ref{thm: 2-3 formulas} and Corollary \ref{thm:density23}. We then prove Theorem \ref{thm:density} using the result of Serre from Section \ref{sec:modular forms}. Finally, we use the properties of D'Arcais polynomials from Section \ref{sec: D'Arcais} as well as results on $\ell$-core congruences \cite{GarvanKimStanton} to prove Theorem \ref{thm:congruences}.

\section*{Acknowledgements}
The authors are deeply grateful to Ken Ono for advising this project and to Eleanor McSpirit for her many valuable comments and suggestions. The authors also thank Hasan Saad, William Craig, and the anonymous referee who reviewed our paper for their helpful comments. The authors were participants in the 2022 UVA REU in Number Theory. They are grateful for the support of grants from the National Science Foundation 
(DMS-2002265, DMS-2055118, DMS-2147273), the National Security Agency (H98230-22-1-0020), and the Templeton World Charity Foundation.

\section{~Nuts and Bolts~}\label{nuts and bolts}

In this section, we first recall a result of Nakamura on the McKay numbers for $S_n$, which we will use for the proof of Theorem \ref{thm:general pl}. We then recall important facts about D'Arcais polynomials and prove a divisibility property of the polynomials that we will use to prove Theorem $\ref{thm:congruences}$. In the final part of the section, we state standard results on the modularity of eta-quotients, which are essential to the proof of Theorem \ref{thm:density}.

\subsection{Nakamura's Generating Function}\label{sec: Nakamura}
In \cite{Nakamura}, Nakamura provides generating functions for McKay numbers and Macdonald numbers for a set of finite groups. In particular, he proves an infinite product expansion for the generating function of the $p_\ell(a;n)$, which correspond to McKay numbers for symmetric groups. We define \begin{equation}F_\ell(x;q) := \sum_{a = 0}^\infty P_\ell(a;q)x^a = \sum_{n=0}^\infty \sum_{a=0}^\infty p_\ell(a;n)x^aq^n.\end{equation}
We then have the following:

\begin{theorem}[Theorem 3.5, \cite{Nakamura}]\label{Nakamura}
Let $[k]_\ell := (\ell^k-1)/(\ell-1)$. If $\ell$ is prime, then 
$$F_\ell(x;q) = \prod_{k=0}^\infty \prod_{n=1}^\infty \frac{(1-x^{\ell n[k]_\ell}q^{\ell^{k+1}n})^{\ell^{k+1}}}{(1-x^{n[k]_\ell}q^{\ell^kn})^{\ell^k}}.$$
\end{theorem}

We will use this generating function to prove the closed formulas for $p_\ell(a;n)$ in Section \ref{sec: proofs}.

\subsection{D'Arcais Polynomials}\label{sec: D'Arcais}

In their study of the Seiberg-Witten theory in \cite{NO}, Nekrasov and Okounkov prove a closed formula for D'Arcais polynomials in terms of partition hook lengths. To set notation, let $\mathcal{P}$ be the set of all partitions and let $\mathcal{H}(\lambda)$ be the multiset of hook lengths of a partition $\lambda$. Then for any complex number $z$, we have (Formula 6.12 of \cite{NO})
    $$\prod_{m=1}^\infty (1-q^m)^{z-1}=\sum_{\lambda \in \mathcal{P}}q^{|\lambda|} \prod_{h \in \mathcal{H}(\lambda)}\left(1-\frac{z}{h^2}\right).$$

\noindent We note that this is equivalent to following formula for the D'Arcais polynomial $f(r,n)$ for any nonnegative integer $n$ and complex number $r$:
    \begin{equation}\label{eqn:NO D'Arcais} f(r,n)=\sum_{\lambda \vdash n} \prod_{h \in \mathcal{H}(\lambda)}\left(1-\frac{r+1}{h^2}\right).\end{equation}
Since a partition of $n$ has exactly $n$ hooks, we have that $r$ appears to degree $n$ in each term of the sum. Therefore, $f(r,n)$ is a polynomial in $r$ of degree at most $n$.

In order to prove Theorem \ref{thm:congruences}, we will need other properties of the D'Arcais polynomials that are not immediately apparent from (\ref{eqn:NO D'Arcais}).
Therefore, we prove an alternative closed formula for the D'Arcais polynomials. To this end, define $D_\lambda$ to be the set of distinct parts of $\lambda$ and let $l(\lambda)$ be the total number of parts of $\lambda$. For $\lambda_i \in D_\lambda$, we write $m(\lambda_i)$ to denote the multiplicity of $\lambda_i$ in $\lambda$. We then have the following theorem.

\begin{theorem}\label{thm:D'Arcais}
If $r$ is a complex number and $n$ is a nonnegative integer, then
    $$f(r;n) = \sum_{\lambda \vdash n}(-1)^{l(\lambda)} \prod_{\lambda_j \in D_\lambda} \binom{r}{m(\lambda_j)},$$
where
$$\binom{r}{m(\lambda_j)} := \frac{r(r-1) \ldots (r-m(\lambda_j)+1)}{m(\lambda_j)!}.$$
\end{theorem}

\begin{proof}
By the binomial theorem, we have
\begin{align*}
    \prod_{n=1}^\infty(1-q^n)^r &= \prod_{n=1}^\infty \sum_{k=0}^\infty \binom{r}{k}(-q^n)^k.
\end{align*}
Let $\mathcal{Z} = \{(z_0, z_1,...) \in \bigoplus_{i\in \mathbb{N}}\mathbb{Z}_{\ge 0}\}$. Then expanding the above product we have
\begin{align*}
    \prod_{n=1}^\infty(1-q^n)^r &= \sum_{\vec{z}\in \mathcal{Z}} \prod_{j=1}^\infty \binom{r}{z_j} (-q^j)^{z_j} \\
    &= \sum_{\vec{z} \in \mathcal{Z}} (-1)^{\sum_{j=0}^\infty z_j}q^{\sum_{j=0}^\infty j z_j} \prod_{j=1}^\infty \binom{r}{z_j}.
\end{align*}
We can view $\mathcal{Z}$ as the set of partitions $\lambda$ with $z_j$ parts of size $j$. For a given $\vec{z}$ equivalent to partition $\lambda$ of $a$, we have 
\begin{align*}
    &\sum_{j=1}^\infty z_j = l(\lambda), \qquad \sum_{j = 1}^\infty j z_j = a, \qquad \prod_{j=1}^\infty \binom{r}{z_j} = \prod_{\lambda_j \in D_\lambda} \binom{r}{m(\lambda_j)}.
\end{align*}
Therefore, we find that
\begin{align*}
    \prod_{n=1}^\infty(1-q^n)^r &= \sum_{a=0}^\infty \sum_{\lambda \vdash a} (-1)^{l(\lambda)} q^a \prod_{\lambda_j \in D_\lambda} \binom{r}{m(\lambda_j)}\\
    &=\sum_{a=0}^\infty q^a \sum_{\lambda \vdash a} (-1)^{l(\lambda)} \prod_{\lambda_j \in D_\lambda} \binom{r}{m(\lambda_j)}.
\end{align*}
\end{proof}

From this formula, we note that $f(r;n)$ is a polynomial in $r$ of degree exactly $n$, since the degree of each term in the sum is determined by $l(\lambda)$ and there is exactly one partition of $n$ with $n$ parts. We also use this formula to show the following property of the polynomials. 

\begin{corollary} \label{cor: D'Arcais}
If $\ell$ is prime, $n$ and $j$ are positive integers, and $\ord_\ell(n) = b < j$, then 

$$f(-\ell^j; n) \equiv 0 \pmod{\ell^{j-b}}.$$
\end{corollary}

\begin{proof}
    Fix a partition $\lambda \vdash n$. Since $\ord_\ell(n) = b$, there exists $\lambda_k \in D_\lambda$ such that $\ord_\ell(m(\lambda_k)) \le b$. We have that
    \begin{align*}
        \binom{-\ell^j}{m(\lambda_k)} &= \pm\frac{(\ell^j+m(\lambda_k)-1)!}{(\ell^j-1)! m(\lambda_k)!}\\
        &=\pm \frac{(m(\lambda_k)+\ell^j)(m(\lambda_k+\ell^j-1))\cdot \ldots \cdot (\ell^j+1) \cdot \ell^j}{(m(\lambda_k)+\ell^j) \cdot m(\lambda_k)!}.
    \end{align*}
    Note that $m(\lambda_k)!$ divides $(m(\lambda_k)+\ell^j)(m(\lambda_k+\ell^j-1))\cdot \ldots \cdot (\ell^j+1)$ and $\ord_\ell(m(\lambda_k)+\ell^j)\le b$. Therefore ${\binom{-\ell^j}{m(\lambda_k)}  \equiv 0 \pmod{\ell^{j-b}}}$. Since $\binom{-\ell^j}{m(\lambda_i)}$ is an integer for any $\lambda_i \in D_\lambda$, we then have that have that $\prod_{\lambda_j \in D_\lambda}\binom{-\ell^j}{\lambda_j} \equiv 0 \pmod{\ell^{j-b}}$. The result follows by Theorem \ref{thm:D'Arcais}.
\end{proof}

This result will be key in establishing the modulus of the congruences in Theorem \ref{thm:congruences}.

\subsection{Modular Forms}\label{sec:modular forms} 

In this section, we state some important facts from the theory of modular forms which we will use to prove Theorem \ref{thm:density}. For a detailed discussion of these topics, see Chapter 1 of \cite{CBMS}.

First recall that $\SL_2(\mathbb{Z})$ acts on the upper half plane $\mathcal{H}$ by linear fractional transformations. In particular, we write $$\gamma z := \frac{az + b}{cz+d}; \hspace{.5cm} \gamma = \left(\begin{matrix}
        a & b \\
        c & d
        \end{matrix}\right) \in \SL_2(\mathbb{Z}). $$

A \emph{weakly holomorphic modular form} of weight $k \in \mathbb{Z}$ and Nebentypus character $\chi$ on a congruence subgroup $\Gamma$ of $\SL_2(\mathbb{Z})$ is a function that is holomorphic on the upper half plane $\mathcal{H}$, whose poles, if any, are supported on the cusps of $\Gamma$, and which satisfies the corresponding \emph{modular transformation law}: for $\gamma = \left(\begin{matrix}
        a & b \\
        c & d
        \end{matrix}\right) \in \Gamma$ and $z \in \mathcal{H}$, we have $$f(\gamma z) = \chi(d)(cz +d)^k f(z).$$
If $f$ satisfies the above conditions and is holomorphic at the cusps of $\Gamma$, we say $f$ is a \emph{holomorphic modular form} and write $f \in M_k(\Gamma, \chi)$. 

In this paper, we will focus on particular congruence subgroups of $\SL_2(\mathbb{Z})$. Namely, for an integer $N > 0$, we define $$ \Gamma_0(N) := \left\{\left(\begin{matrix}
        a & b \\
        c & d
        \end{matrix}\right) \in \SL_2(\mathbb{Z}) : c\equiv 0 \pmod{N} \right\}. $$

The primary function that we will make use of is Dedekind's eta-function $\eta(z)$, which is given by the infinite product 

 $$ \eta(z) := q^\frac{1}{24} \prod_{n=1}^\infty (1 - q^n),$$
 where $q := e^{2\pi i z}$. It is well known that $\eta(z)$ is also a modular form, albeit of a slightly more general type than discussed above. In particular, $\eta(z)$ admits a modular transformation law characterized by the relations 
    \begin{align*}
        \eta(z + 1) = e^{\pi i/12}\cdot\eta(z), \\
        \eta(-1/z) = \sqrt{-i z} \cdot\eta(z).
    \end{align*}
It turns out that $\eta(z)$ is a modular form of weight $1/2$. Although $\eta(z)$ will be central to the proof of Theorem \ref{thm:density}, we will study products of powers of $\eta(z)$ which result in integer weight modular forms. For the reader interested in half integer weight modular forms more generally, see \cite{CBMS}.

A function $f(z)$ is called an \emph{eta-quotient} if it is of the form $$f(z) = \prod_{\delta | N} \eta(\delta z)^{r_\delta}$$
where $N$ is a positive integer and $r_\delta \in \mathbb{Z}$ for each $\delta$.
In the proof of Theorem \ref{thm:density}, we will construct an eta-quotient whose Fourier expansion is the generating function of $p_\ell(0;n)$. We can then apply properties of eta-quotients to the generating function of $p_\ell(0;n)$. In particular, we recall the following theorem, which we will use to show that the eta-quotient we construct is an integer weight modular form. 

\begin{theorem}[p. 18 of \cite{CBMS}]\label{thm: mod check}
    If $f(z) = \prod_{\delta | N} \eta(\delta z)^{r_\delta}$ is an eta-quotient with $k = \frac{1}{2}\sum_{\delta | N} r_\delta \in \mathbb{Z}$, with the additional properties that 
    $$\sum_{\delta | N} \delta r_\delta \equiv 0 \pmod{24} $$
    and
    $$\sum_{\delta | N} \frac{N}{\delta}r_\delta \equiv 0 \pmod{24}, $$
    then $f(z)$ satisfies $$f\left(\frac{az + b}{cz+d}\right) = \chi(d)(cz+d)^kf(z).$$
    for every $\left(\begin{matrix}
        a & b \\
        c & d
        \end{matrix}\right) \in \Gamma_0(N)$. Here the character $\chi$ is defined by $\chi(d) := \left(\frac{(-1)^ks}{d}\right)$ where $s := \prod_{\delta | N} \delta^{r_\delta}$.
\end{theorem}

The above theorem gives a necessary condition for an eta-quotient to satisfy a suitable modular transformation law. Since $\eta(z)$ is clearly nonvanishing on $\mathcal{H}$, in order to show that an eta-quotient is holomorphic on the congruence subgroup $\Gamma_0(N)$, it suffices to show that it is holomorphic at the cusps of  $\Gamma_0(N)$. To this end, we make use of the following theorem, which gives the order of vanishing of an eta-quotient $f(z)$ at the cusps of $\Gamma_0(N)$. 

\begin{theorem}[p. 18 of \cite{CBMS}]\label{thm:ord vanish}
    Let $c,d,$ and $N$ be positive integers with $d | N$ and $\gcd(c,d) = 1$. If $f(z)$ is an eta-quotient satisfying the conditions of Theorem \ref{thm: mod check} for $N$, then the order of vanishing of $f(z)$ at the cusp $\frac{c}{d}$ is
    $$\frac{N}{24}\sum_{\delta | N} \frac{\gcd(d, \delta)^2r_\delta}{\gcd(d, \frac{N}{d})d\delta}. $$
\end{theorem}

 In order to prove Theorem \ref{thm:density}, we will write $P_\ell(0;q)$ as an eta-quotient. We then employ Theorems \ref{thm: mod check} and \ref{thm:ord vanish} to prove that this eta-quotient is an integer weight homolomorphic modular form on the congruence subgroup  $\Gamma_0(\ell)$. Finally, we show that $p_\ell(0;n)$ is almost always divisible by $m$ for any integer $m$ by applying the following theorem, which is a generalization of a result of Serre about the $\mod m$ lacunarity of integer weight modular forms \cite{Serre}.

\begin{theorem}[p. 43 of \cite{CBMS}]\label{thm: mod exp dens}
    Suppose that $f(z)$ is an integer weight modular form in $M_k(\Gamma_0(N), \chi)$ and has Fourier expansion $$f(z) = \sum_{n=0}^\infty a(n)q^n, $$
    where $a(n) \in \mathbb{Z}$ for all $n$. Then there exists a constant $\alpha > 0$ such that $$\#\{n \le X: a(n) \ne 0 \pmod{m}\} = O\left(\frac{X}{(\log{X})^\alpha}\right). $$
\end{theorem}

In the case of Theorem \ref{thm:density}, the Fourier coefficients $a(n)$ will be the $p_\ell(0;n)$. Therefore, Theorem \ref{thm: mod exp dens} bounds the number of $p_\ell(0;n)$, with $n<X$, that are not divisible by an integer $m$. When we combine this with Theorem \ref{thm:general pl}, we will obtain Theorem \ref{thm:density}.

\section{Proofs of Theorems}\label{sec: proofs}

In this section we prove the results stated in Section 1.

\subsection{Proof of Theorem \ref{thm:general pl}}\label{proof general pl}
Separating the $k=0$ term from the product in Theorem \ref{Nakamura}, using the generating function for $\ell$-core partitions, and then writing Nakamura's infinite product as a product of D'Arcais polynomials, we have that
\begin{align*}
    F_\ell(x;q) &= \prod_{n=1}^{\infty} \frac{\left(1-q^{\ell n}\right)^\ell}{\left(1-q^n\right)} \cdot
    \prod_{k=1}^\infty \prod_{n=1}^\infty \frac{\left(1-x^{\ell n [k]_\ell}q^{\ell^{k+1} n}\right)^{\ell^{k+1}}}{\left(1-x^{n [k]_\ell}q^{\ell^{k} n}\right) ^{\ell^{k}}} \\
    &= P_\ell(0;q) \cdot  \prod_{k=1}^\infty \prod_{n=1}^\infty \frac{\left(1-x^{\ell n [k]_\ell}q^{\ell^{k+1} n}\right)^{\ell^{k+1}}}{\left(1-x^{n [k]_\ell}q^{\ell^{k} n}\right) ^{\ell^{k}}}\\
    &= P_\ell(0;q) \cdot \prod_{k=1}^\infty\left(\sum_{m=0}^\infty \left[ f(\ell^{k+1};m)x^{m\ell[k]_\ell}q^{m\ell^{k+1}}\right] \cdot \sum_{h=0}^\infty \left[f(-\ell^k;h) x^{h[k]_\ell}q^{h\ell^k}\right]\right) \\
    &= P_\ell(0;q) \cdot \prod_{k=1}^\infty\left( \sum_{h=0}^\infty \left[\sum_{m=0}^\infty f(\ell^{k+1};m)\cdot f(-\ell^k;h)x^{(m\ell+h)[k]_\ell}q^{(m\ell+h)\ell^{k}}\right]\right).
\end{align*}

\noindent Reindexing with $z =m\ell+h$, we then have that
\begin{align*}
    F_\ell(x;q)
    &= P_\ell(0;q) \cdot \prod_{k=1}^\infty\left( \sum_{z=0}^\infty \left[\sum_{m=0}^{\lfloor z /\ell \rfloor} f(\ell^{k+1};m)\cdot f(-\ell^k;z-m\ell)x^{z [k]_\ell}q^{z\ell^{k}}\right]\right)\\
    &= P_\ell(0;q) \cdot \prod_{k=1}^\infty\left( \sum_{z=0}^\infty \left[ x^{z [k]_\ell}q^{z\ell^{k}} \sum_{m=0}^{\lfloor z/\ell \rfloor} f(\ell^{k+1};m) \cdot f(-\ell^k;z-m\ell)\right]\right).
\end{align*}
Now, let $$\mathcal{Z}_\ell(a) := \bigg\{(z_1, z_2,...) \in \bigoplus_{i\in \mathbb{N}}\mathbb{Z}_{\ge 0} \; :\;\; \sum_{k\ge 1}z_k [k]_\ell = a \bigg\}.$$ Expanding the product in the above formula for $F_\ell(x;q)$, we have that 
\begin{align*}
    F_\ell(x;q)
    &= P_\ell(0;q) \cdot \sum_{a=0}^\infty \sum_{\vec{z} \in \mathcal{Z}_\ell(a)} \left(\prod_{k=1}^\infty x^{z_k [k]_\ell} q^{z_k \ell^k} \sum_{m=0}^{\lfloor z_k/\ell \rfloor}\left[f(\ell^{k+1};m)\cdot f(-\ell^k; z_k - m\ell)\right]\right) \\
    &= P_\ell(0;q) \cdot \sum_{a=0}^\infty x^a \sum_{\vec{z} \in \mathcal{Z}_\ell(a)} \left(\prod_{k=1}^\infty q^{z_k \ell^k} \sum_{m=0}^{\lfloor z_k/\ell \rfloor}\left[f(\ell^{k+1};m)\cdot f(-\ell^k; z_k - m\ell)\right]\right).
\end{align*}
The result follows. \qed

\subsection{Proof of Corollary \ref{thm:pl polys}}
Let $a<\ell$. Then for $k>1$, we have that $[k]_\ell>a$, and it follows that $\mathcal{Z}_\ell(a) = \left\{(a,0,0,...)\right\}$. Applying this to Theorem \ref{thm:general pl} (noting that $\lfloor a/\ell\rfloor = 0$) gives
\begin{align*}
    P_\ell(a;q) &= P_\ell(0;q) \cdot q^{a\ell}\cdot \left[f(\ell^{k+1};0)\cdot f(-\ell^k; a)\right] \cdot \left[\prod_{k=1}^\infty f(\ell^{k+1};0)\cdot f(-\ell^k;0) \right]\\
    &= P_\ell(0;q) \cdot q^{a\ell}\cdot f(-\ell^k; a).
\end{align*}
The second equality follows from the fact that $f(r;0)=1$ for any $r$. \qed

\subsection{Proof of Theorem \ref{thm: 2-3 formulas}}

From Corollary \ref{thm:pl polys}, we have that 
\begin{align*}
    P_2(1;q) &= 2 q^2 \cdot P_2(0; q), \\
    P_3(1;q) &= 3 q^3 \cdot P_3(0; q), \\
    P_3(2;q) &= 9 q^6 \cdot P_3(0; q).
\end{align*}
Using the identity for $P_2(0;q)$ and $P_3(0;q)$ from \cite{GranOno}, we then have that 

\begin{align*}
    P_2(1;n) &= 2 q^2 \cdot \sum_{m=0}^\infty q^{\frac{m^2+m}{2}}, \\
    P_3(1;n) &= 3 q^3 \cdot \sum_{m=0}^\infty  \left(\sum_{d | 3m+1}\left(\frac{d}{3}\right) \right) q^m, \\
    P_3(2;n) &= 9 q^6 \cdot \sum_{m=0}^\infty  \left(\sum_{d | 3m+1}\left(\frac{d}{3}\right) \right) q^m.
\end{align*}
Theorem \ref{thm: 2-3 formulas} follows after accounting for the shift in exponents of $q$.  \qed

\subsection{Proof of Corollary \ref{thm:density23}}

In the case that $\ell = 2$, we get from (\ref{2core}) that

\begin{equation}\label{2asymptotic}
    \#\{n \le X : p_2(0;n) \ne 0\} = \#\{k \in \mathbb{Z}^+: (k^2 + k)/2\le X\} = O(\sqrt{X}).
\end{equation}

\noindent If $\ell = 3$, we claim that
\begin{equation}\label{3asymptotic}
    \#\{n \le X : p_3(0;n) \ne 0\} = \#\bigg\{n \le X : \sum_{d\mid 3n+1} \left(\frac{d}{3}\right) \ne 0\bigg\} = O\left(\frac{X}{\sqrt{\log X}}\right).
\end{equation}
The equality above immediately follows from (\ref{3core}). To justify the $O$-bound, as noted on p. 334 of \cite{GranOno}, one employs elementary sieve theory. 

Turning to the case of general $a$, we define $g_2(X) := \sqrt{X}$ and $g_3(X) := \frac{X}{\sqrt{\log X}}$ so that we have ${\#\{n \le X: p_\ell(0;n) \ne 0\} = O(g_\ell(X))}$ for $\ell = 2,3$. For any $n < 0$, we let $p_\ell(0; n) = 0$. It then follows from Theorem \ref{thm:general pl} that
$$p_\ell(a;n) = \sum_{\vec{z}\in \mathcal{Z}_\ell(a)} p_\ell\Big(0;n-\sum_{k=1}^\infty z_k\ell^k\Big)\prod_{k=1}^\infty \sum_{m=1}^{\lfloor z_k/\ell\rfloor}f(\ell^{k+1};m)\cdot f(-\ell^k;z_k-m\ell),$$
where $\mathcal{Z}_\ell(a)$ is defined as in (\ref{Zdef}) and Section \ref{proof general pl}. For an arbitrary nonnegative integer $a$, $p_\ell(a;n) \ne 0$ only if $p_\ell(0; n - \sum_{k\ge 1}z_k \ell^k) \ne 0$ for at least one $\vec{z}\in \mathcal{Z}_\ell(a)$. Therefore, 
\begin{align*}
    \#\{n \le X : p_\ell(a;n) \ne 0\} 
    &\le \#\Bigg\{n \le X : \sum_{\vec{z}\in \mathcal{Z}_\ell(a)} p_\ell\Big(0;n- \sum_{k\geq1}z_k \ell^k\Big) \ne 0\Bigg\} \\ 
    &\le  \sum_{\vec{z}\in \mathcal{Z}_\ell(a)}\#\Big\{n \le X : p_\ell\Big(0;n- \sum_{k\geq1}z_k \ell^k\Big) \ne 0\Big\} \\ 
    &\le  \sum_{\vec{z}\in \mathcal{Z}_\ell(a)}\#\{n \le X : p_\ell(0;n) \ne 0\} \\     
    &= O_a(g_\ell(X)).
\end{align*}
The last line follows from (\ref{2asymptotic}, \ref{3asymptotic}) and the fact that $\mathcal{Z}_\ell(a)$ is finite. This proves Corollary \ref{thm:density23}. \qed

\subsection{Proof of Theorem \ref{thm:density}}

Let $\ell$ be an odd prime. We first note that the theorem is proven in the case when $\ell = 3$ by Corollary \ref{thm:density23}. Therefore, we may assume that $\ell \ge 5$. We begin by defining \begin{equation} f_\ell(z) := \frac{\eta^\ell(\ell z)}{\eta(z)}.\end{equation}

Taking $N = \ell$, by Theorem \ref{thm: mod check}, we get that $f_\ell(z)$ is a modular form of weight $\frac{\ell-1}{2}$ with Nebentypus character $\chi = \left( \frac{\bullet}{\ell} \right)$. Also by Theorem \ref{thm:ord vanish}, for $c,d$ positive integers such that $(c,d) = 1$ and $d | \ell$, the order of vanishing of $f_\ell (z)$ at a cusp $\frac{c}{d}$ is $$\frac{\ell}{24}\left(\frac{\ell(\gcd(d, \ell)-1)}{\gcd(d, \frac{\ell}{d})d\ell} \right).$$
Since $\gcd(d, \ell) \ge 1$, we get that the order of vanishing of $f_\ell(z)$ is nonnegative at any cusp $\frac{c}{d}$ of $\Gamma_0(\ell)$. Therefore, since $\frac{\ell-1}{2} \in \mathbb{Z}$ for odd primes $\ell$, we get that $f_\ell(z) \in M_{\frac{\ell-1}{2}}(\Gamma_0(\ell), \left( \frac{\bullet}{\ell} \right))$. 

We can also expand the eta quotient of $f_\ell(z)$ using the definition of $\eta(z)$ to get $$f_\ell(z) = q^{\frac{\ell^2-1}{24}} \prod_{n=1}^\infty \frac{(1-x^{\ell n})^\ell}{1-x^n}.$$
Because $\ell \ge 5$, we have that $\ell^2 \equiv 1 \pmod{24}$ and thus $\frac{\ell^2 -1}{24} \in \mathbb{Z}$.
Since $$\sum_{m = 0}^\infty p_\ell(0;n)q^n =  \prod_{n=1}^\infty \frac{(1-x^{\ell n})^\ell}{1-x^n},$$ we get that $f_\ell(z) = \sum_{n = 0}^\infty p_\ell(0;n)q^{\frac{\ell^2-1}{24}+n}$. Now we may apply Theorem \ref{thm: mod exp dens} to $f_\ell(z)$ to get that there exists a constant $\alpha_{\ell, m} > 0$ such that
\begin{equation}\label{eq: pl0 dens}
{\#\{n \le X: p_\ell(0;n) \ne 0 \pmod{m}\} = O\left(\frac{X}{(\log{X})^{\alpha_{\ell,m}}}\right)}.
\end{equation}
By the same argument as in the proof of Corollary \ref{thm:density23}, this implies for arbitrary $a$ that $${\#\{n \le X: p_\ell(a;n) \ne 0 \pmod{m}\} = O_a\left(\frac{X}{(\log{X})^{\alpha_{\ell,m}}}\right)}.$$
\qed

\subsection{Proof of Theorem \ref{thm:congruences}}
Recall from Section \ref{sec: D'Arcais} that if $\ell \nmid a$, then $f(-\ell, a) \equiv 0 \pmod{\ell}$ and $f(-\ell^2;a)\equiv 0 \pmod{\ell^2}$ . We also recall from \cite{GarvanKimStanton} that for any positive integer $n$
\begin{equation}\label{eq: l core cong}
    \begin{split}
        p_5(0;5^m n - (5^2-1)/24) &\equiv 0 \pmod{5^m}, \\
        p_7(0;7^m n - (7^2-1)/24) &\equiv 0 \pmod{7^m}, \\
        p_{11}(0;11^m n - (11^2-1)/24) &\equiv 0 \pmod{11^m}.
    \end{split}
\end{equation}

\noindent We now proceed by considering the five cases of Theorem \ref{thm:congruences} separately.

(1) By definition of $\delta(a,\ell)$ and Corollary \ref{thm:pl polys}, we have that, for $a<\ell$, 
\begin{align*}
    p_\ell(a; \ell^m n -\delta(a, \ell)) &= f(-\ell; a) \cdot p_\ell(0; \ell^m n - (\ell^2-1)/24).
\end{align*}
Part 1 of the Theorem follows from (\ref{eq: l core cong}) and the fact that $f(-\ell; a) \equiv 0 \pmod{\ell}$.

(2) By Theorem \ref{thm:general pl} we have that 
\begin{align*}
    p_\ell(\ell; \ell^m n -\delta(\ell, \ell)) 
    &= (f(-\ell; \ell)-\ell^2 ) \cdot p_\ell(0; \ell^m n - (\ell^2-1)/24).
\end{align*}
Part 2 of Theorem \ref{thm:congruences} follows from (\ref{eq: l core cong}).

(3) In the case that $\ell+1 < a < 2\ell$, from Theorem \ref{thm:general pl} we have that 
\begin{align*}
    p_\ell(a; \ell^3 n -\delta(a; \ell))
    =& \; (f(-\ell; a)-\ell^2 f(-\ell; a-\ell)) \cdot p_\ell(0; \ell^3 n - (\ell^2-1)/24) \\ &+ f(-\ell; a -\ell-1)\cdot \ell^2 \cdot p_\ell(0; \ell^3 n - (\ell^2-1)/24 + \ell).
\end{align*}
Note that by (\ref{eq: l core cong}) and the fact that $f(-\ell; a) \equiv f(-\ell;a-\ell)\equiv f(-\ell;a-\ell-1) \equiv 0 \pmod{\ell}$, both terms are divisible by $\ell^4$. Part 3 of Theorem \ref{thm:congruences} follows.

(4) In the case that $a = 2\ell$, from Theorem \ref{thm:general pl} we have that 
\begin{align*}
    p_\ell(2\ell; \ell^4 n -\delta(2\ell; \ell))
    =& \; (f(-\ell; 2\ell)-\ell^2 f(-\ell; \ell) + f(\ell^2; 2)) \cdot p_\ell(0; \ell^4 n - (\ell^2-1)/24) \\ &+ f(-\ell; \ell-1)\cdot \ell^2\cdot p_\ell(0; \ell^4 n - (\ell^2-1)/24 + \ell).
\end{align*}
Note that by (\ref{eq: l core cong}) and the fact that $f(-\ell;\ell-1) \equiv 0 \pmod{\ell}$, both terms are divisible by $\ell^4$. Part 4 of Theorem \ref{thm:congruences} follows. 

(5) In the case that $0<a<\ell(\ell+1)$, from Theorem \ref{thm:general pl} we have that
\begin{align*}
    p_\ell(a;\ell^3 n - \delta(a;\ell)) =& \; p_\ell\Big(0;\ell^3 n-\frac{\ell^2-1}{24}\Big) \cdot \sum_{m=0}^{\lfloor a/\ell\rfloor}f(\ell^2;m)f(-\ell;a-m\ell) \\
    &+ \sum_{j=1}^{\lfloor a/(\ell+1) \rfloor} \Bigg[p_\ell \Big(0;\ell^3 n - \ell j- \frac{\ell^2-1}{24}\Big)
    \cdot f(-\ell^2; j) \\
    & \qquad\qquad  \cdot \sum_{m=0}^{\lfloor (a-j(\ell+1))/\ell \rfloor} f(\ell^2;m) f(-\ell;a-j(\ell+1) - m\ell)\Bigg].
\end{align*}
We have that $f(-\ell^2; j) \equiv 0 \pmod{\ell^2}$ for $1 \le j \le \ell-1$. By (\ref{eq: l core cong}) we also have that 

\noindent$p_\ell(0;\ell^3 n-(\ell^2-1)/24) \equiv 0 \pmod{\ell^3}$ and that $p_\ell(0;\ell^3 n - \ell j- (\ell^2-1)/24) \equiv 0 \pmod{\ell}.$   Part 5 of Theorem \ref{thm:congruences} follows. \qed

\bibliography{References.bib}{}
\bibliographystyle{abbrv}
\end{document}